\documentclass[11pt,reqno]{amsart}
\usepackage{amsopn}
\usepackage{amssymb, amscd}
\usepackage{latexsym,bm}
\usepackage{amsmath}
\usepackage[english]{babel}
\usepackage{mathrsfs}
\usepackage{graphicx}
\usepackage{calligra}

\usepackage{tikz}
\usepackage{subfigure}

\usepackage{pgfplots}

\usepackage{MnSymbol}

\newcommand{\nc}{\newcommand}

\nc{\reales}{\mathbb{R}}
\nc{\SR}{\mathbb{R}}
\nc{\SN}{\mathbb{N}}

\nc{\sa}{S_\alpha}      \nc{\pa}{P_\alpha}     \nc{\qa}{Q_\alpha}
\nc{\ha}{H_\alpha}

\nc{\lp}{L^p (\omega)}                               \nc{\normlp}{\raisebox{-0.1cm}{\scriptsize{$L^p (\omega)$}}}
\nc{\lpp}{L^p (\omega^{-p})}                         \nc{\normlpp}{\raisebox{-0.1cm}{\scriptsize{$L^p (\omega^{-p})$}}}
\nc{\lloc}{L_{loc}^{1} (\mathbb{R}^n)}
\nc{\linfty}{L^{\infty} (\omega^{-1})}
\nc{\bmogamma}{BMO^{\gamma}(\omega)}                 \nc{\normbmogamma}{\raisebox{-0.1cm}{\scriptsize{$BMO^{\gamma}(\omega)$}}}
\nc{\bmodelta}{BMO^{\delta}(\omega)}                 \nc{\normbmodelta}{\raisebox{-0.1cm}{\scriptsize{$BMO^{\delta}(\omega)$}}}
\nc{\bmgcero}{BM^{\gamma}_{0} (\omega)}              \nc{\normbmgcero}{\raisebox{-0.1cm}{\scriptsize{$BM^{\gamma}_{0} (\omega)$}}}
\nc{\bmdcero}{BM^{\delta}_{0} (\omega)}              \nc{\normbmdcero}{\raisebox{-0.1cm}{\scriptsize{$BM^{\delta}_{0} (\omega)$}}}

\nc{\Rhp}{RH (p)} \nc{\rhp}{RH_0 (p)} \nc{\rhpp}{RH_0 (p')}
\nc{\deta}{D_{\eta}} \nc{\detacero}{D_{0,\eta}}
\nc{\detap}{D_{\eta'}}

\nc{\essinf}{\hbox{ess}\inf}
\nc{\supp}{\operatorname{Supp}}

\nc{\charac}{\raisebox{1pt}{$\chi$}}
\nc{\characmed}{\raisebox{0.07cm}{$\chi$}}
\nc{\characdos}{\raisebox{2pt}{$\chi$}}

\nc{\vsdos}{\vspace{.2cm}}
\nc{\vsuno}{\vspace{1cm}}

\nc{\normluno}{\raisebox{-0.1cm}{\scriptsize{$L^1 (\omega)$}}}
\nc{\normlinfty}{\raisebox{-0.1cm}{\scriptsize{$L^{\infty} (\omega^{-1})$}}}
\nc{\normbmo}{\raisebox{-0.1cm}{\scriptsize{$BMO(\omega)$}}}

\nc{\normbmcero}{\raisebox{-0.1cm}{\scriptsize{$BM_0 (\omega)$}}}
\nc{\normbmocero}{\raisebox{-0.1cm}{\scriptsize{$BMO_0 (\omega)$}}}
                     \nc{\normbmogcero}{\raisebox{-0.1cm}{\scriptsize{$BMO^{\gamma}_{0} (\omega)$}}}
                    \nc{\normbmodcero}{\raisebox{-0.1cm}{\scriptsize{$BMO^{\delta}_{0} (\omega)$}}}

\numberwithin{equation}{section}

\theoremstyle{plain}
\newtheorem{theorem}{Theorem}[section]
\newtheorem{lemma}[theorem]{Lemma}
\newtheorem{corollary}[theorem]{Corollary}

\theoremstyle{definition}
\newtheorem{definition}[theorem]{Definition}

\theoremstyle{remark}
\newtheorem{remark}[theorem]{Remark}

\renewcommand {\epsilon} {{\varepsilon}}

\newcommand{\R}{\mathbb R}

\newcommand {\al} {{\alpha}}
\newcommand {\dt} {{\delta}}
\newcommand {\Dt} {{\Delta}}
\newcommand {\e} {{\varepsilon}}
\newcommand {\ga} {{\gamma}}
\newcommand {\Ga} {{\Gamma}}

\renewcommand {\O} {{\Omega}}
\renewcommand {\phi} {{\varphi}}
\newcommand {\la} {{\lambda}}
\newcommand {\La} {{\Lambda}}

\newcommand {\bp} {{\bar p}}
\newcommand {\bC} {{\mathscr{C}}}

\newcommand{\Ba}[1]{\begin{array}{#1}}
\newcommand{\Ea}{\end{array}}
\newcommand{\Be}{\begin{equation}}
\newcommand{\Ee}{\end{equation}}
\newcommand{\Bea}{\begin{eqnarray}}
\newcommand{\Eea}{\end{eqnarray}}
\newcommand{\Beas}{\begin{eqnarray*}}
\newcommand{\Eeas}{\end{eqnarray*}}
\newcommand{\Benu}{\begin{enumerate}}
\newcommand{\Eenu}{\end{enumerate}}

\renewcommand {\mid} {{\,\,\,\colon\,\,\,}}
\newcommand {\mand} {{\quad\mbox{and}\quad}}

\newcommand{\sline}{{\smallskip

\noindent}}

\newcommand{\T}{{\mathscr T}}
\newcommand{\Th}{{\T_{\rm H}}}
\newcommand{\Tou}{{\T_{\rm OU}}}
\newcommand{\sT}{ \mathscr{T}^{\rm sub} }
\newcommand{\sTh}{{\sT_{\rm H}}}
\newcommand{\sTou}{{\sT_{\rm OU}}}
\newcommand{\supT}{ \mathscr{T}^{\rm sup} }
\newcommand{\supTh}{{\supT_{\rm H}}}
\newcommand{\supTou}{{\supT_{\rm OU}}}
\newcommand{\tx}{\tilde{x}}
\newcommand{\tit}{\tilde{t}}
\newcommand{\ts}{\tilde{s}}
\newcommand{\ty}{\tilde{y}}

\newcommand{\Etz}{E_{t_0}}

\newcommand {\ProofEnd} {
             \begin{flushright} \vskip -0.2in $\Box$ \end{flushright}}

\def\Xint#1{\mathchoice
{\XXint\displaystyle\textstyle{#1}}%
{\XXint\textstyle\scriptstyle{#1}}%
{\XXint\scriptstyle\scriptscriptstyle{#1}}%
{\XXint\scriptscriptstyle\scriptscriptstyle{#1}}%
\!\int}
\def\XXint#1#2#3{{\setbox0=\hbox{$#1{#2#3}{\int}$ }
\vcenter{\hbox{$#2#3$ }}\kern-.6\wd0}}

\def\mint{\Xint-}

\title[Mean value formulas]{Mean value formulas for  Ornstein-Uhlenbeck and Hermite temperatures}

\author{Guillermo Flores and Gustavo Garrig\'os$^*$}
\address{G. J. Flores, CIEM-FaMAF, Universidad Nacional de C\'ordoba, Av. Medina Allende s/n, Ciudad Universitaria, CP:X5000HUA C\'ordoba, Argentina}
\email{gflores@famaf.unc.edu.ar}
\address{G. Garrig\'os, Departamento de Matem\'aticas, Universidad de Murcia, 30100, Espinardo, Murcia, Spain}
\email{gustavo.garrigos@um.es}
\thanks{First author partially supported by grants from CONICET, SeCyT (Universidad Nacional de C\'ordoba) and
        Universidad Nacional del Litoral (Argentina). Second author partially supported by grants MTM2016-76566-P, MTM2017-83262-C2-2-P 
				and Programa Salvador de Madariaga PRX18/451 from Micinn (Spain), and grant 20906/PI/18  from Fundaci\'on S\'eneca (Regi\'on de Murcia, Spain).
				}
\thanks{$^*$\emph{Corresponding author:} Gustavo Garrig\'os, \emph{Email:} gustavo.garrigos@um.es} 



\begin{document}

\maketitle

\begin{abstract}
We obtain explicit mean value formulas for the solutions of the diffusion equations associated with the 
 Ornstein-Uhlenbeck and Hermite operators. From these, we derive various useful properties, such as maximum principles, uniqueness theorems and Harnack-type inequalities.
\end{abstract}


\section{Introduction}\label{introduction}
\indent Given an open set $E\subset \SR^{n}\times\SR$, we denote by $C^{2,1} (E)$ the set of real-valued functions 
$u(x,t)$ on $E$ such that the partial
derivatives $\partial_t u $  and $\partial^2_{x_i,x_j} u $, $1\leq i, j\leq n$,
all exist and are continuous on $E$. \\
\indent We say that $u \in C^{2,1} (E)$ is a \emph{temperature in $E$}, denoted $u\in\T(E)$, if
\begin{align*}
{\partial_t u} = \Dt u =\sum_{i = 1}^{n} {\partial^2_{x_i,x_i} u}, \quad \hbox{ $(t,x)\in E$},
\end{align*}
that is, if $u(x,t)$ solves the classical heat equation in the domain $E$.

\indent In this paper we shall be interested in two variants of the above PDE, namely the \emph{Ornstein-Uhlenbeck heat equation}, given by
\begin{align} \label{OU}
{\partial_t U} = 
\Big(\Dt-2x\cdot\nabla\Big)U,
\end{align}
and the \emph{Hermite heat equation}, given by
\begin{align} \label{Hermite}
{\partial_t U} = 
\Big(\Dt-|x|^2\Big)U.
\end{align}
Functions $U \in C^{2,1} (E)$ satisfying  \eqref{OU} or \eqref{Hermite} will be called OU-temperatures or H-temperatures, and the corresponding classes will be denoted by  $\Tou(E)$ and $\Th(E)$, respectively.

The goal of this paper is to provide \emph{explicit} mean value formulas for functions in $\Tou$ and  $\Th$,
which are similar to the mean value formulas for classical temperatures in $\T$ due to Watson; see \cite{watson73,watsonbook} or \cite[p. 53]{evans}. Namely, we shall prove the following theorem, which seems to be new in these settings.

\begin{theorem}\label{main}
There exists a family of bounded open sets $\Xi(x,t;r)$, for $(x,t)\in\R^{n+1}$ and $r>0$, and positive kernels $K^{\rm OU}_{x,t}(y,s)$ and $K^{\rm H}_{x,t}(y,s)$ in $C^\infty\big(\R^n\times(-\infty,t)\big)$, such that the following properties are equivalent:
\Benu
\item[(a)] $U\in\Tou(E)$

\medskip

\item[(b)] for every $\;{\overline{\Xi}}(x,t;r)\subset E$ it holds
\[
U(x,t)= \frac{1}{(4\pi r)^{\frac n2}} \iint_{\Xi(x,t;r)} U(y,s)\, K^{\rm OU}_{x,t} (y,s)\, dy ds.\label{Uxt}
\]
\Eenu
The same equivalence holds for $U\in\Th(E)$ if  $K^{\rm OU}_{x,t}$ is replaced by $K^{\rm H}_{x,t}$.
\end{theorem}

The explicit expressions for the kernels and ``balls'' are given below; see \eqref{KOU}, \eqref{KH} and Figure \ref{fig1}. 
 We emphasize that the family $\Xi(x,t;r)$ is the same for both, OU and H temperatures, and only 
the kernels change slightly.
This result may be used as a starting point to establish several classical properties,
such as strong maximum principles, uniqueness theorems or Harnack inequalities, for functions in $\Tou$ and $\Th$. 

In particular, we shall prove below the following uniqueness theorem with an optimal unilateral growth condition, which seems new in this generality.

\begin{theorem} \label{th_uni1}
Let $0<T_0\leq \infty$, and let $U\in\Tou(\SR^n\times(0,T_0))$ be continuous in $\SR^n\times[0,T_0)$
and with $U(x,0)\equiv 0$. Suppose additionally that for some $A>0$ it holds
\Be
U(x,t)\leq \,A \,e^{|x|\,p(|x|)},\quad \forall\;|x|\geq1,\;\forall\;t\in(0,T_0)
\label{uniTa}
\Ee
where $p(r)$ is a positive continuous function, such that $r^\ga p(r)$ is non-decreasing for some $\ga\geq0$, and
\[
\int_1^\infty\frac{dr}{p(r)}=\infty.
\]  
Then, necessarily $U\equiv0$ in $\SR^n\times[0,T_0)$.

Conversely, for every such $p(r)$ with $\int_1^\infty\frac{dr}{p(r)}<\infty$, there exists a function $U\in\Tou(\SR^n\times\SR)$ with $U(x,0)\equiv0$, $U\not\equiv0$
and 
\[
|U(x,t)|\leq e^{|x|p(|x|\vee 1)},\quad \forall\;x\in\SR^n,\;t\in(0,\infty).
\]
Finally, both statements hold as well when $\Tou$ is replaced by $\Th$.
\end{theorem}

\section{Notation and main results}

\subsection{A transference principle}\label{S_trans}
Our aproach will be based on known transference formulas between the classes $\T$, $\Tou$ and $\Th$ which we describe next. Given $U \in C^{2,1} (E)$, we consider the  transformation
\Be
V(x,t) := e^{-nt} e^{-|x|^2 /2} U(x,t).\label{VU}\Ee
It is easily verified that
\Be
\Big(\partial_t-\Dt+|x|^2\Big)[V(x,t)]\,=\,e^{-nt}e^{\frac{-|x|^2}2}\,\Big(\partial_t-\Dt+2x\cdot\nabla\Big)[U(x,t)],
\label{edpV}\Ee
from which we conclude that \begin{align} \label{H-sii-OU}
U\in \Tou(E) \quad \Longleftrightarrow \quad V\in\Th(E).
\end{align}
This elementary transformation has recently been used in various contexts regarding Hermite operators; see e.g. \cite{Abu, AST,GHSTV}.

\

We  next give a transformation which relates the classes $\T$ and $\Tou$. 
It is suggested by the Mehler formula expression, and it is more or less implicit
in the early works in the topic \cite{hille, muck, sjogren}.

\

\indent Let $ \varphi : \SR^{n+1} \to \SR^n \times (-\infty, 1/4) $ be the $C^\infty$-diffeomorphism
given by
\begin{equation}
\label{phi}
\varphi(x,t) = \Big( \frac{x}{e^{2t}}, \frac{1-e^{-4t}}{4} \Big),
\end{equation}
whose inverse takes the form
\begin{equation*}
\varphi^{-1} (y,s) = \Big(\; \frac{y}{\sqrt{1 - 4s}}\;, \;\tfrac{1}{4} \log \tfrac{1}{1-4s} \;\Big).
\end{equation*}
Observe that $ \varphi $ is increasing and preserves the positivity in the $t$-variable. Next, we define the transformation
$T : C (\varphi(E)) \to C (E)$ given by
\begin{equation}
U(x,t):=T u (x,t) = \; u \big(\varphi(x,t)\big).
\label{TuU}
\end{equation}
Then we have the following elementary relation.
\begin{lemma}
If $U=Tu\in C^{2,1}(E)$, then
\Be
\Big(\partial_t-\Dt+2x\cdot\nabla_x\Big)[U(x,t)]\,=\,e^{-4t}\,T\Big[\big(\partial_s-\Dt_y\big)u\Big].
\label{edpU}\Ee
\end{lemma}
\begin{proof}
From $U(x,t)=u\big(xe^{-2t},(1-e^{-4t})/4\big)$ one easily obtains
\Beas
\partial_tU(x,t)& =& e^{-4t}u_s\big(\phi(x,t)\big)-2e^{-2t}\,x\cdot(\nabla_yu)\big(\phi(x,t)\big)\\
\partial_{x_i}U(x,t)& =& e^{-2t}\,u_{y_i}\big(\phi(x,t)\big)\\
\partial^2_{x_ix_i}U(x,t)& = & e^{-4t}u_{y_iy_i}\big(\phi(x,t)\big).
\Eeas
From here the expression \eqref{edpU} follows immediately.
\end{proof}
As a consequence we conclude that,
\begin{align} \label{Li}
u\in \T(\varphi(E)) \quad \Longleftrightarrow \quad U=Tu\in\Tou(E).
\end{align}


\subsection{Classical mean value formulas}

We recall the mean value formula for classical temperatures in $\T$.
More generally, we state the result for \emph{subtemperatures}, that is 
for functions $u\in C^{2,1}(E)$ such that
\[
u_t\leq \Dt u, \quad \mbox{in $E$},
\]
which we shall denote by $u\in \sT(E)$. 
 For each $(x, t) \in \reales^n\times\SR$ and $r > 0$, we define the \emph{heat ball} $\O(x,t;r)$, with ``center'' $(x,t)$ and ``radius'' $r$, by
\begin{equation*}
\Omega(x,t;r) = \Big\{ (y,s) \in \reales^n \times\R \;\colon\; 
 \Phi(x-y,t-s) > \tfrac1{(4\pi r)^{\frac n2}} \Big\}, 
\end{equation*}
where
\begin{equation*}
\Phi (x,t) =  \left\{ 
        \begin{tabular}{ll}
        	$ \frac{e^{-|x|^2 / 4t}}{(4 \pi t)^{n/2}} $ & if $x \in \reales^n$, $t > 0$\\
            $0$ & if $x \in \reales^n$, $t \leq 0$ \\
        \end{tabular}
\right.
\end{equation*}
is the usual fundamental solution of the heat equation. 
Observe that $(y,s)\in\O(x,t;r)$ iff 
\[
|y-x|^2<2n(t-s)\ln\tfrac r{t-s},\quad s\in(t-r,t),
\]
so $\O(x,t;r)$ is an open convex set, axially symmetric about
the line $\{x\} \times \reales$, and with $(x,t)$ lying at the top boundary and $(x, t-r)$ at the bottom boundary;
see e.g. \cite[p. 52]{evans} or \cite[p. 2]{watsonbook}. Heat balls are also translation invariant, in the sense that\[
\O(x,t;r)\,=\,(x,t)+\O(0,0;r).
\] 
The next theorem, due to N. Watson \cite{watson73},  is known as mean value property for the heat equation; see also  
 \cite[Theorem 1.16]{watsonbook}.

\begin{theorem} \label{MV-Heat}
Let $u \in C^{2,1} (E)$.  If $u \in \sT(E)$ then 
\begin{equation}
u(x,t) \leq \frac{1}{(4\pi r)^\frac n2} \iint_{\Omega (x,t;r)} u(y,s)\, \frac{|x - y|^2}{4(t - s)^2}\, dy ds
\label{uleq}
\end{equation}
provided $\overline{\Omega} (x,t;r) \subset E$.
Conversely if for every $(x, t) \in E$ and every $\epsilon > 0$ there exists some $r < \epsilon$ such
that \eqref{uleq} holds, then $u \in \sT(E)$.
\end{theorem}
In particular, the following characterization holds.

\begin{corollary}
Let $u \in C^{2,1} (E)$. Then $u\in\T(E)$ if and only if
\Be\label{clasMV}
u(x,t) = \frac{1}{(4\pi r)^\frac n2} \iint_{\Omega (x,t;r)} u(y,s)\, \frac{|x - y|^2}{4(t - s)^2}\, dy ds
\Ee
for all heat balls $\overline{\Omega} (x,t;r) \subset E$.
\end{corollary}

\subsection{Mean value formulas in $\sTou$ and $\sTh$} \label{MeanValue}
We first 
define the corresponding notion of ``heat ball'', by using the transformation $\phi$ in \eqref{phi}. Namely, given $(x_0,t_0)\in\R^{n+1}$ 
and $r>0$ we let
\Be
\Xi(x_0,t_0;r):=\,\phi^{-1}\Big(\O\big(\phi(x_0,t_0);r\big)\Big).
\label{Xiball}
\Ee
Equivalently, $(x,t)\in \Xi(x_0,t_0;r)$ iff
\[
\Big|\frac x{e^{2t}}- \frac {x_0}{e^{2t_0}}\Big|^2<\tfrac n2\,(e^{-4t}-e^{-4t_0})\,\ln \tfrac{4r}{e^{-4t}-e^{-4t_0}},
\]
so this set is now axially symmetric with respect to the curve 
$\Ga(x_0,t_0)=\{(x_0\exp\big(2(t-t_0)\big), t)\mid t\in\R\}$, with the point $(x_0,t_0)$ lying at the top boundary.
 Also, unlike the classical heat balls, these $\Xi$-balls are no longer translation invariant. See Figure \ref{fig1} below for drawings in various situations.

\

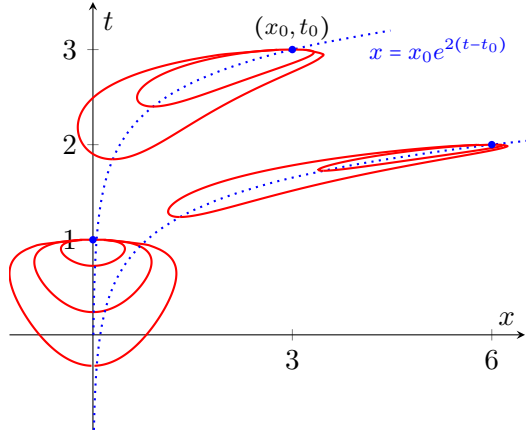
\begin{figure}[h]
 \centering

\begin{tikzpicture} 
\begin{axis}[axis lines=middle, xmax=6.5,ymax=3.5,ymin=-1, 
             xlabel={$x$}, ylabel={$t$}, xtick={0,3,6}, ytick={0,1,2,3}]


\addplot[red,  thick,  domain=0.01:200, smooth, samples=1000 ] 
		   ( { (x+1)^(-1/2) *(0+((ln(200/x))*x*0.5)^(0.5)) }, { 1-0.25*ln(x+1) } )   ;
\addplot[red,  thick,  domain=0.01:200, smooth, samples=1000] 
       ( { (x+1)^(-1/2) *(0-((ln(200/x))*x*0.5)^(0.5)) }, { 1-0.25*ln(x+1) } )   ;
			
\addplot[red,  thick, domain=0.01:20, smooth, samples=200]  
       ( { (x+1)^(-1/2) *(0+((ln(20/x))*x*0.5)^(0.5)) }, { 1-0.25*ln(x+1) } )   ;
\addplot[red,  thick,  domain=0.01:20, smooth, samples=200]  
       ( { (x+1)^(-1/2) *(0-((ln(20/x))*x*0.5)^(0.5)) }, { 1-0.25*ln(x+1) } )   ;
			
\addplot[red,  thick, domain=0.01:2, smooth, samples=200]  
       ( { (x+1)^(-1/2) *(0+((ln(2/x))*x*0.5)^(0.5)) }, { 1-0.25*ln(x+1) } ) ;
\addplot[red,  thick, domain=0.01:2, smooth, samples=200]  
       ( { (x+1)^(-1/2) *(0-((ln(2/x))*x*0.5)^(0.5)) }, { 1-0.25*ln(x+1) } ) ;
\addplot[red,  thick, domain=-0.2:0.2] (x,1);
\node[circle,fill=blue, inner sep=1pt] at (axis cs:0,1) {};


\addplot[thick, dotted,   domain=0:3.2,smooth,variable=\t,blue] 
			                  ({3*exp(2*(\t-3))}, {\t});

\addplot[thick, domain=0.01:100,smooth,variable=\s,red, samples=500]  
	                     ( { (\s+1)^(-1/2) *(3+((ln(100/\s))*\s*0.5)^(0.5)) }, { 3-0.25*ln(\s+1) } )   ;
\addplot[thick, domain=0.01:100,smooth,variable=\s,red, samples=500]  
                      ( { (\s+1)^(-1/2) *(3-((ln(100/\s))*\s*0.5)^(0.5)) }, { 3-0.25*ln(\s+1) } )   ;
											
\addplot[thick, domain=0.01:10,smooth,variable=\s,red, samples=400]  
	                     ( { (\s+1)^(-1/2) *(3+((ln(10/\s))*\s*0.5)^(0.5)) }, { 3-0.25*ln(\s+1) } )   ;
\addplot[thick, domain=0.01:10,smooth,variable=\s,red, samples=400]  
                      ( { (\s+1)^(-1/2) *(3-((ln(10/\s))*\s*0.5)^(0.5)) }, { 3-0.25*ln(\s+1) } )   ;
\addplot[thick, red, no marks] coordinates {  (2.8,2.998) (3,3) (3.18,2.998)}  ;	
\node[circle,fill=blue, inner sep=1pt] at (axis cs:3,3) {};

\addplot[thick, dotted,   domain=-1:3.2,smooth,variable=t,blue] 
			                  ({6*exp(2*(t-2)}, {t});			
\addplot[thick, domain=0.01:2, smooth,variable=\s,red, samples=500]  
	                     ( { (\s+1)^(-1/2) *(6+((ln(2/\s))*\s*0.5)^(0.5)) }, { 2-0.25*ln(\s+1) } )   ;
\addplot[thick, domain=0.01:2, smooth,variable=\s,red, samples=500]  
                      ( { (\s+1)^(-1/2) *(6-((ln(2/\s))*\s*0.5)^(0.5)) }, { 2-0.25*ln(\s+1) } )   ;
											
\addplot[thick, domain=0.01:20, smooth,variable=\s,red, samples=400]  
	                     ( { (\s+1)^(-1/2) *(6+((ln(20/\s))*\s*0.5)^(0.5)) }, { 2-0.25*ln(\s+1) } )   ;
\addplot[thick, domain=0.01:20, smooth,variable=\s,red, samples=400]  
                      ( { (\s+1)^(-1/2) *(6-((ln(20/\s))*\s*0.5)^(0.5)) }, { 2-0.25*ln(\s+1) } )   ;
																						
\addplot[thick, red, no marks] coordinates {  (5.8,1.995) (6,2) (6.15,1.998)}  ;	
\node[circle,fill=blue, inner sep=1pt] at (axis cs:6,2) {};

\node[above] at (axis cs:3,3) {{\footnotesize $(x_0,t_0)$}};
\node[right,blue] at (axis cs:4,3) {{\footnotesize $x=x_0 e^{2(t-t_0)}$}};
			
\end{axis}
\end{tikzpicture}

\caption{Hermite heat balls $\Xi(x_0,t_0;r)$ of various centers and radii.}\label{fig1}
 \end{figure}

\

We now state a slightly more general version than 
Theorem \ref{main}, which is also valid for subtemperatures. As before, we say that $U\in C^{2,1}(E)$ belongs to $\sTou(E)$ if
\Be
\partial_t U\leq \big(\Delta -2x\cdot\nabla\big) U\quad\mbox{in $E$},\label{OUleq}\Ee
and that $U\in \sTh(E)$ if 
\Be
\partial_t U\leq \big(\Delta -|x|^2\big) U\quad\mbox{in $E$}.\label{Hleq}\Ee
With the notation in \eqref{VU} and \eqref{TuU}, it follows from \eqref{edpV} and \eqref{edpU} that
\Be \label{Lis}
 u\in \sT(\varphi(E))  \;\; \Longleftrightarrow \;\; U\in\sTou(E)
 \;\;\Longleftrightarrow  \;\; V\in\sTh(E) .
\Ee

\

\begin{theorem} \label{MV-OU}
Let $U \in C^{2,1} (E)$. If $U \in \sTou(E)$ and $(x,t)\in E$ then
\begin{equation}
U(x,t) \leq \frac{1}{(4\pi r)^{\frac n2}} \iint_{\Xi(x,t;r)} U(y,s)\, K^{\rm OU}_{x,t} (y,s) dy ds\label{Uxt}
\end{equation}
provided $\;\overline{\Xi}(x,t;r) \subset E$, where
\Be
K^{\rm OU}_{x,t} (y,s) =\,8 e^{-2(n+2)s}
                \;
						\frac{|x{e^{-2t}} - y{e^{-2s}}|^2}
								{|e^{-4t} - e^{-4s}|^2} .
\label{KOU}\Ee
Conversely if for every $(x, t) \in E$ and every $\epsilon > 0$ there exists some $r < \epsilon$ such
that \eqref{Uxt} holds, then $U \in \sTou(E)$.
\end{theorem}

\begin{remark}
When $U\in\Tou(E)$, then  the above theorem, applied to the functions $U$ and $-U$, easily implies Theorem \ref{main}.
\end{remark}

The version of Hermite subtemperatures takes the following form.

\begin{theorem} \label{MV-Hermite}
Let $V \in C^{2,1} (E)$. If $V \in \sTh(E)$ and $(x,t)\in E$ then
\begin{equation}
V(x,t) \leq \frac{1}{(4\pi r)^{\frac n2}} \iint_{\Xi(x,t;r)} V(y,s)\, K^{\rm H}_{x,t} (y,s) dy ds\label{Vxt}
\end{equation}
provided $\;\overline{\Xi}(x,t;r) \subset E$, where
\Be
K^{\rm H}_{x,t} (y,s) =\,e^{(s-t)n}\,e^{\frac{|y|^2-|x|^2}2}\, K^{\rm OU}_{x,t} (y,s).
\label{KH}\Ee
Conversely if for every $(x, t) \in E$ and every $\epsilon > 0$ there exists some $r < \epsilon$ such
that \eqref{Vxt} holds, then $V \in \sTh(E)$.
\end{theorem}

\subsection{Proof of Theorems \ref{MV-OU} and \ref{MV-Hermite}}\label{ProofMV}
 \begin{proof}[Proof of Theorem \ref{MV-OU}] Assume that $U\in\sTou(E)$. By \eqref{Lis} we know that $u = T^{-1} U\in \sT(\varphi(E))$, so we can apply Theorem \ref{MV-Heat} at the point $(\tilde{x}, \tilde{t})=\phi(x,t)$ to obtain
\begin{equation}
u(\tilde{x}, \tilde{t}) \leq \frac{1}{(4\pi r)^\frac n2} \iint_{\Omega (\tilde{x},\tilde{t};r)}
            u(\tilde{y}, \tilde{s}) K_{\tx,\tit}(\ty,\ts)\, d\tilde{y} d\tilde{s},
\label{utxtt}
\end{equation}
with $ K_{\tx,\tit}(\ty,\ts)=2^{-1}{|\tilde{x} - \tilde{y}|^2}/{(\tilde{t} - \tilde{s})^2}$ and 
provided $\overline{\Omega} (\tilde{x},\tilde{t};r) \subset \varphi(E)$. Now, 
making the change of variables
\[
(\tilde{y},\tilde{s})=\Big(\frac y{e^{2s}},\frac{1-e^{-4s}}4\Big)=\phi(y,s),\]
whose jacobian is given by\[
\left|\frac{\partial(\tilde{y},\tilde{s})}{\partial(y,s)}\right|=\,e^{-2(n+2)s},
\]
and using the identity $U = u\circ\phi$,
one easily obtains the formula in \eqref{Uxt} with
\Be
K_{x,t}^{\rm OU}(y,s)=e^{-2(n+2)s}\,K_{\phi(x,t)}\big(\phi(y,s)\big),
\label{KKphi}
\Ee 
which agrees with \eqref{KOU}.
The converse follows by a completely similar argument.
\end{proof}
\begin{proof}[Proof of Theorem \ref{MV-Hermite}] Again, \eqref{Lis} implies that $U(x,t)=e^{ t n}e^{\frac{|x|^2}2} V(x,t)$ belongs to $\sTou(E)$, and therefore we can use the formula \eqref{Uxt}.
From here, elementary operations lead  to \eqref{Vxt}. 
\end{proof}


\section{Consequences: Maximum principles}\label{MaxPr}

\indent As it is standard in potential theory, we shall employ the mean value formulas in \eqref{Uxt} and \eqref{Vxt} to easily derive  maximum principles (in strong form) for functions in $\sTou$ and $\sTh$. These results are known in the literature for more general parabolic pdes, but typically with different proofs; see e.g. \cite[ch 3.3]{protter} or
\cite[ch 2.2]{friedman}. Here we follow the approach in \cite[ch 2.3]{evans}.

\

We shall use the general set terminology in \cite{friedman, protter}.
Given an open set $E$ and a point $(x_0, t_0) \in E$, we denote by $E_{t_0} = \{(x,t) \in E : t < t_0\}$. Also, we denote by
$\Lambda (x_0, t_0,E)$ the set of points $(x,t)$ that are lower than $(x_0, t_0)$ relative to $E$,
in the sense that there is a polygonal path $\gamma \subset E$ joining $(x_0, t_0)$ to $(x,t)$, along which the temporal
variable is \emph{strictly decreasing}. By a polygonal path we mean a path which is a union of finitely many line segments. 
So, if $E=\O\times(0,\infty)$, for some connected open set $\O\subset\R^n$, then $\Lambda(x_0,t_0,E)=E_{t_0}=\O\times(0,t_0)$, for all $x_0\in\O$ and $t_0>0$. For drawings in other situations, see Figure \ref{fig2} below or  \cite[p. 169]{protter}.

\begin{figure}[h]
 \centering
{\begin{tikzpicture}[scale=2]
\draw[->] (0.0,-1.1) -- (0.0,1.1) node[left] {$t$};
\draw[dotted] (3,0.5)-- (-0.1,0.5) node[left] {{\footnotesize $t_0$}};

\draw (0.2,0.9)--(2.7,0.9)--(2.35,0.3)--(2,0.6)--(1,-1)--(0.2,0.9); 
\draw[fill, opacity=0.3] (0.37,0.5)--(1,-1)--(1.685,0.096)--(1.5,0.096)--(1.33,0.5); 

\draw (1.25,0.7)--(1.75,0.7)--(1.5,0.1)--(1.25,0.7);
\draw[fill=white] (0.75,0.3)--(1.25,0.3)--(1,-0.3)--(0.75,0.3);


\fill (1,0.5) circle (1pt);
\node[above] at (1,0.5) {{\tiny $(x_0,t_0)$}};
\node[right] at (0,-0.7) {{\tiny $\La(x_0,t_0,E)$}};
\node[above] at (2,0.9) {{\footnotesize $E$}};

\end{tikzpicture}
}
\caption{The subset $\La(x_0,t_0,E)$ (in grey) of a set $E$.}\label{fig2}
\end{figure}
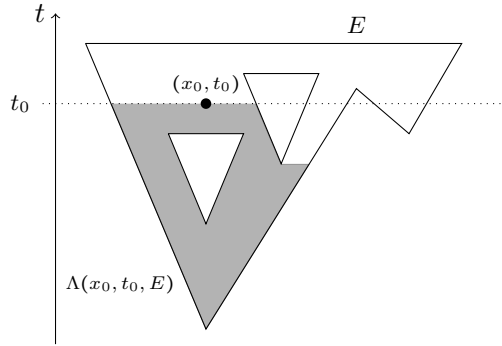

\begin{remark} \label{remark1}
Observe that $\varphi \big( \Lambda (x_0, t_0,E) \big) = \Lambda \big( \varphi (x_0, t_0), \varphi (E)\big)$, since the mapping $\phi$ is monotone in the $t$-variable. In particular, this implies that
$\Xi(x_0,t_0;r) \subset \Lambda(x_0, t_0, E)$,  for all sufficiently small $r > 0 $, since the same property holds for the classical heat balls. 
\end{remark}
\begin{remark} \label{varphi}
Both $\Omega(x_0,t_0;r)$ and $\Xi(x_0,t_0;r)$, 
have the same tangent plane at the point $(x_0,t_0)$, which is normal to the direction $(0,1)$. In particular, every poligonal path $\ga(s)$, with $\ga(0)=(x_0,t_0)$ and with strictly
decreasing temporal variable, has a portion $\ga\big((0,\e]\big)\subset\Xi(x_0,t_0;r)$, for some $\e>0$.
\end{remark}

We can now derive a \emph{strong maximum principle} for Hermite and Ornstein-Uhlenbeck 
subtemperatures. 

\begin{corollary} \label{StrongMaximum} {\bf (Strong maximum principle)}
Assume that $U \in \sTou (E)$ and there exists $(x_0,t_0) \in E$ such that
\begin{equation*}
\sup_{\La(x_0,t_0,E)} \, U\,=\, U(x_0,t_0).
\end{equation*}
Then $U$ is constant in the set $\Lambda(x_0, t_0, E)$. The same holds when $U\in \sTh(E)$
provided that $U(x_0,t_0)\geq 0$.
\end{corollary}
\begin{proof} Let $M:={U(x_0, t_0)}$. Assume first that $U\in\sTh(E)$ and hence that $M\geq0$.
Observe from \eqref{Hleq} that the constant $-M$ also belongs to $\sTh$.
If $r > 0 $ is sufficiently small, then using twice the mean value property \eqref{Vxt} and the positivity of the involved kernel one obtains
\Bea
U(x_0,t_0) & \leq & \frac{1}{(4\pi r)^\frac n2} \iint_{\Xi(x_0,t_0;r)} U(y,s) K^{\rm H}_{x_0,t_0} (y,s) dy ds \nonumber\\
       & \leq & 
                  \frac{1}{(4\pi r)^\frac n2} \iint_{\Xi(x_0,t_0;r)} M \cdot K^{\rm H}_{x_0,t_0} (y,s) dy ds
         \;\leq\; M.\label{Maux1}
\Eea
Thus equality holds in the middle expressions, and hence 
\[
\iint_{\Xi(x_0,t_0;r)} \big[M-U(y,s)\big]\, K^{\rm H}_{x_0,t_0} (y,s)\, dy ds =0,
\]
which implies that $U\equiv M$  within
$\Xi(x_0,t_0;r)$.

\

To prove constancy in all $\La(x_0,t_0,E)$ we follow a standard argument, see e.g. \cite[p.56]{evans}. 
Let $(y_0 ,s_0) \in \Lambda(x_0, t_0, E)$, i.e. $s_0 < t_0$, and let $\gamma$ be a polygonal path
joining $(x_0,t_0)$ to $(y_0, s_0)$ for which the temporal variable is strictly decreasing. Consider
\begin{equation*}
s_1: = \min \Big\{ s \in [s_0,t_0] \mid U|_{\{\gamma\}\cap \R^n\times[s,t_0]}\equiv M \Big\} ,
\end{equation*}
where $ \{ \gamma \} $ is the image set of $\gamma$. The set in the brackets is non-empty, and since $U$ is continuous, the minimum is attained.
Assume $s_1 > s_0$. Then $U(z_1,s_1) = M$ for some point $(z_1,s_1) \in \{ \gamma \} $
and so, as before, $U$ is identically equal to $M$ within $\Xi(z_1,s_1;r)$
for all sufficiently small $r > 0$. But, by Remark \ref{varphi}, the ball
$\Xi(z_1,s_1;r)$
contains $ \{ \gamma \} \cap \reales^n \times [s_1- \epsilon, s_1) $ for some small
$\epsilon$, which is a contradiction. Therefore, $s_1 = s_0$ and then $U\equiv M$ 
on $\{ \gamma \}$. This completes the proof when $U\in \sTh$ and $M\geq0$.

When $U\in\sTou$ one applies exactly the same argument, this time using that
the constant $M\in\Tou$, and hence the last step in \eqref{Maux1} is an equality.
\end{proof}

From the previous result one can obtain versions of the \emph{weak maximum principle}. We state one which is valid for any open set $E\subset\R^{n+1}$ (possibily unbounded).

\begin{corollary} \label{weak1} {\bf (Weak maximum principle)}
Let $U\in \sTou(E)$  and  $t_0 \in \SR$. Assume that for some $A\in\R$ it holds
\Be
\limsup_{{(x,t)\to(x_1,t_1)}\atop{(x,t)\in E_{t_0}}} \, U(x,t)\; \leq \; A,
\quad \forall\; (x_1,t_1)\in \partial_{\rm P} \big(E_{t_0}\big)\cup\{\infty\}
\footnote{In this expression the infinity point only plays a role if $\Etz$ is unbounded.}
\label{limsup}
\Ee
where $\partial_{\rm P} \big(E_{t_0}\big):=\partial \big(E_{t_0}\big)\setminus\big[\{t=t_0\}\cap E\big]$.
Then
\Be
U(x,t)\leq A,\quad \forall\;(x,t)\in E_{t_0}.
\label{supUA}
\Ee
The same holds for $U\in\sTh(E)$ when, additionally, $\sup_{E_{t_0}} U\geq 0$.
\end{corollary}
\begin{proof}
Let $M:=\sup_{E_{t_0}} U $. We must show that $M\leq A$. Pick a sequence of points $\{P_n\}_{n\geq 1}\subset \Etz$ such that $U(P_n)\nearrow M$. There are several cases: 

\sline (i) If $\{P_n\}_{n\geq1}$ is not bounded, then
\Be
M=\lim_{n\to\infty}U(P_n)\leq 
\limsup_{{(x,t)\to\infty }\atop{(x,t)\in E_{t_0}}}U(x,t) \leq A.
\label{MinfA}
\Ee
 
\sline (ii) If $\{P_n\}_{n\geq1}$ is bounded, then passing to a subsequence
we may assume that $P_n$ converges to some $P\in {\overline{\Etz}}$.
Assume first that  $P\in \partial_{\rm P} \Etz$. In this case we directly use \eqref{limsup} to obtain
\Be
M=\lim_{n\to\infty}U(P_n)\leq 
\limsup_{{(x,t)\to P }\atop{(x,t)\in E_{t_0}}}U(x,t) \leq A.
\label{MQA}
\Ee
\sline (iii) Finally, assume that the point $P$ in case (ii) belongs to 
the remaining set, namely  $P\in \Etz$ or $P\in E\cap\{t=t_0\}$. In either case $U(P)=M$ and the strong maximum principle in Corollary \ref{StrongMaximum} shows that\footnote{In the case $U\in\sTh(E)$ this step uses the assumption $M\geq0$ stated after \eqref{supUA}.} $U\equiv M$ in $\La(P,E)$.
If $\La(P,E)$ is unbounded, then we can find a sequence $Q_n\in \La(P,E)$ going to $\infty$, so that
\eqref{MinfA} holds (with $Q_n$ in place of $P_n$), and we would be done. If $\La(P,E)$ is bounded, then we pick any straight line $\ga(s)$ with $\ga(0)=P$ and strictly decreasing in the $t$-variable.
Consider the finite positive number
\[
s^*=\min\{s>0\mid \ga(s)\not\in \La(P,E)\},
\]
and the point $Q=\ga(s^*)$. By construction, $Q$ must belong to $\partial_P(E_{t_0})$. Then if $s_n\nearrow s^*$ we have $Q_n=\ga(s_n)\in \La(P,E)$,
so \eqref{MQA} holds with $P_n,P$ replaced by $Q_n,Q$.
This completes the proof of $M\leq A$ in all cases.
\end{proof}

In the special case of bounded domains $E$ and continuous functions up to the boundary one recovers the following classical statement.

\begin{corollary} \label{WeakMaximum}
Let $E\subset\R^{n+1}$ be open and bounded, and let $U\in C(\overline{E})\cap \sTou(E)$.
Then, for every  $t_0\in \SR$
\Be
\max_{{\overline{\Etz}}} \, U\; = \; \max_{\partial_{\rm P}\Etz} \, U.
\label{maxLap}
\Ee
The same holds when $U\in\sTh(E)$ if, additionally, the left hand-side of \eqref{maxLap} is non-negative.
\end{corollary}
\begin{proof}
Use the previous corollary with $A=\max_{\partial_{\rm P} \Etz}U$.
\end{proof}

\begin{remark}\label{R1}
For band domains $E=\O\times(0,T_0)$, with $\O\subset\R^n$ a bounded open set,  \eqref{maxLap} takes the form
\[
\max_{{\overline{\O}}\times[0,t_0]} \, U\; = \;\max_{({\overline{\O}}\times\{0\})\cup(\partial\O\times[0,t_0])}
\!\!\!U,\quad  \mbox{if $t_0\in (0,T_0]$}.
\]
That is, the maximum of $U$ in ${\overline{\O}}\times[0,t_0]$ is always attained at some point of its \emph{parabolic boundary}; \cite[p. 52]{evans}.
\end{remark}

\

Finally, as in \cite[p. 56]{evans}, we deduce a corollary about infinite propagation speed. Here we say that $U$ is a 
\emph{supertemperature}, denoted $U\in\supTou$, if $-U\in\sTou$ (and likewise for  $\supTh$).

\begin{corollary}{\bf (Infinite propagation)}. Let $E=\O\times(0,T_0)$, with $\O\subset\R^n$ bounded, open and connected.
Let $U\in C(\bar E)$ belong to  the class $\supTou(E)$ or $\supTh(E)$.
Assume that 
\Be
U\geq 0 \mbox{ in $\partial_{\rm P} E$},\mand \exists\;x_0\in \O\mbox{ such that }U(x_0,0)>0.
\label{infprop}
\Ee
Then $U(x,t)>0$ at all $(x,t)\in E$.
\end{corollary}
\begin{proof}
Suppose, for contradiction, that there exists $P=(x_1,t_1)\in E$ such that $U(P)\leq 0$. Call $V=-U$, which is a subtemperature.
Then $\sup_E V\geq V(P)\geq 0$, so we can use Corollary \ref{WeakMaximum} (in the special case of Remark \ref{R1})
to deduce that
\[
\sup_E V =\sup_{\partial_{\rm P} E} V \leq 0.\]
But then 
\[
V(P)=\sup_E V=0,\]
which in turn, by Corollary \ref{StrongMaximum} and the connectivity of $\O$, implies that $V(x,t)\equiv 0$ if $t\in (0,t_1)$, $x\in\O$.
This contradicts the second condition in \eqref{infprop}.
\end{proof}

%
%

\

\section{Uniqueness of solutions} \label{Uniq}

We next derive some uniqueness results for the Cauchy problem when $E=\SR^n\times(0,T_0)$ with $0<T_0\leq \infty$.
Let $L$ denote one of the operators
\Be
\Dt,\quad \Dt-2x\cdot\nabla\quad\mbox{or}\quad \Dt-|x|^2.
\label{ops}
\Ee
Given $f\in C(\SR^n\times(0,T_0))$ and $g\in C(\SR^n)$, 
we say that $u(x,t)$ is a \emph{classical solution} of
\Be
\left\{
\begin{tabular}{ll}
        	$u_t = Lu +f$ & in $\SR^n\times (0,T_0)$ \\
            $u(x,0) = g$ & on $\SR^n$,
        \end{tabular}
\right.
\label{P}
\Ee
whenever $u$ belongs to $C\big(\SR^n\times[0,T_0)\big)\cap C^{2,1}\big(\SR^n\times(0,T_0)\big)$ and satisfies \eqref{P}.
When $L=\Dt$, a well-known condition to ensure uniqueness is 
\Be
\sup_{0<t<T_0}|u(x,t)|\leq Ae^{c|x|^2},\quad x\in\SR^n,
\label{Ty}
\Ee
for some $A,c>0$; see e.g. \cite[Thm 2.3.7]{evans}. On the other hand, for every $\e>0$ there exist infinitely many solutions of \eqref{P} with $f=g=0$ and
\Be
\sup_{0<t<T_0}|u(x,t)|\leq e^{|x|^{2+\e}},\quad x\in\SR^n.
\label{Tye}
\Ee 
When $L$ is one of the last two operators in \eqref{ops}, the condition \eqref{Ty} is also sufficient for uniqueness.
This could be proved from the above maximum principles, but is also a special case of results for general
 parabolic equations; see e.g. \cite[Theorem 2.4.10]{friedman}.
Here we attempt to replace \eqref{Ty} by an \emph{optimal} growth condition, which goes back to the work of T\"acklind \cite{Tac}.  

\begin{definition} 
Let $p(r)$ be a positive and continuous function for $r\geq1$. We define the class $\mathscr{C}(p)$ as the set of all 
 $u\in C(\SR^n\times[0,T_0))$ such that
\Be
\sup_{0<t<T_0}|u(x,t)|\leq A e^{|x|p(|x|)},\quad \forall\;|x|\geq1,
\label{Ta}
\Ee
for some constant $A>0$ (which may depend on $u$ and $T_0$). 
\end{definition}

Below we also consider the function \[
\bp(r):=\inf_{s\geq r} p(s),\]
that is, the largest non-decreasing minorant of $p(r)$. Note that $\bp(r)$ is also continuous, and that $\bp(r)=p(r)$ when $p$ is non-decreasing.

\begin{theorem}[T\"acklind \cite{Tac}] \label{heat-T}
Let $p(r)$ be positive, continuous for $r\geq1$. Then, the following are equivalent for the operator $L=\Dt$

\sline (a) the equation \eqref{P} has at most one classical solution
$u$ in the class $\mathscr{C}(p)$ 

\sline (b) the function $p$ satifies
\Be
\int_1^\infty\frac{dr}{\bp(r)}=\infty.
\label{intbp}
\Ee
\end{theorem}

In practice, one would often use \eqref{Ta} with functions $p(r)$ which are eventually increasing, such as 
$p(r)=r(\log r)(\log\log r)\cdots (\log^{\mbox{{\tiny$(k)$}}}\!\!r)$, for $r\geq e^{\udots^{e}}$. 
In that case, the condition \eqref{intbp} is equivalent to
\Be
\int_\al^\infty\frac{dr}{p(r)}=\infty,\quad \mbox{for some $\al\geq1$}.
\label{intp}
\Ee
In general, however, when $p$ oscillates, \eqref{intbp} may not necessarily imply \eqref{intp}.
We now give a simple criterion which guarantees the equivalence of both conditions. 
The proof 
is a slight modification of T\"acklind original argument \cite[p. 16]{Tac}; see also \cite[p. 396]{Hayne}.

\begin{lemma}\label{Lem1}
Let $p(r)$ be positive and continuous for $r\geq1$. Assume also that 
for some $\ga\geq0$ the function  $r^\ga p(r)$ is non-decreasing. Then \eqref{intbp} is equivalent to \eqref{intp}.
\end{lemma}
\begin{proof}
Since $\bp(r)\leq p(r)$, we only need to show that \eqref{intbp} $\Rightarrow$ \eqref{intp}. We first notice that, if $\liminf_{r\to\infty}p(r)<\infty$
then \eqref{intp} always holds. Indeed, in such case we can find a sequence $r_j\nearrow\infty$ such that \Be
r_j\geq 2 r_{j-1}\mand \sup p(r_j)\leq C.
\label{auxp1}
\Ee
Then, for $r\in(r_{j-1},r_j)$ we have $r^\ga p(r)\leq r_j^\ga p(r_j)$ and hence
\Be
\int_{r_{j-1}}^{r_j}\frac{dr}{p(r)} \geq \int_{r_{j-1}}^{r_j}\frac{r^\ga dr}{r_j^\ga p(r_j)} \geq c_\ga\frac{r_j^{\ga+1}-r_{j-1}^{\ga+1}}{C\,r^\ga_j}\geq c'_\ga r_j\to\infty.
\label{auxp7}
\Ee
Thus, we may assume that $\lim_{r\to\infty}p(r)=\infty$. This implies that
the set $\{r\mid \bp(r)=p(r)\}$
is necessarily unbounded (because $\min_{[j,\infty)}p$ is always attained at some $r_j\geq j$, and then $p(r_j)=\bp(r_j)$).
We now follow the construction in \cite[p. 16]{Tac}. Let $\ell_0=1$, and let $\ell_j$ be the smallest real number such that
\Be
\ell_j\geq 2\ell_{j-1}\mand \bp(\ell_j)=p(\ell_j).
\label{ellj}
\Ee 
Clearly,
\[
\int_{\ell_{j-1}}^{\ell_j}\frac{dr}{\bp(r)}\geq \frac{\ell_j-\ell_{j-1}}{\bp(\ell_j)}\geq \frac{\ell_j/2}{p(\ell_j)}.
\]
On the other hand, since by construction $\bp(r)=\bp(\ell_j)$ when $r\in(2\ell_{j-1},\ell_j)$, we also have
\[
\int_{\ell_{j-1}}^{\ell_j}\frac{dr}{\bp(r)}  =  \int_{\ell_{j-1}}^{2\ell_{j-1}}+\int_{2\ell_{j-1}}^{\ell_j}\ldots
\leq 
\frac{\ell_{j-1}}{\bp(\ell_{j-1})}+\frac{\ell_j-2\ell_{j-1}}{\bp(\ell_j)} \leq \frac{\ell_{j-1}}{p(\ell_{j-1})}+\frac{\ell_j}{p(\ell_j)}.
\]
Thus, we conclude that
\Be
\int_1^\infty\frac{dr}{\bp(r)}=\infty \quad \iff\quad \sum_{j=1}^\infty \frac{\ell_j}{p(\ell_j)}=\infty.
\label{intp=sum}
\Ee
Finally, if we further assume that $r^\ga p(r)$ is non-decreasing, then
\Be
\int_{\ell_{j-1}}^{\ell_j}\frac{dr}{p(r)}= \int_{\ell_{j-1}}^{\ell_j}\frac{r^\ga dr}{r^\ga p(r)}
\geq \frac{(\ell_j^{\ga+1}-\ell_{j-1}^{\ga+1})/(\ga+1)}{\ell_j^\ga p(\ell_j)}\geq c_\ga \frac{\ell_j}{p(\ell_j)},
\label{intg}
\Ee
so from  \eqref{intp=sum} we obtain \eqref{intp}. 
\end{proof}

For our proof below we need one more auxiliary result.

\begin{lemma}
\label{Lem2}
Let $p(r)$ be as in the statement of Lemma \ref{Lem1}, and for fixed $\la>0$  define $p_1(r)=p(r)+\la r$.
Then
\Be
\int_1^\infty\frac{dr}{\bp_1(r)}=\infty \quad \iff\quad \int_1^\infty\frac{dr}{\bp(r)}=\infty.
\label{intp1}
\Ee
\end{lemma}
\begin{proof}
Note that $p_1(r)$ also satisfies the conditions of Lemma \ref{Lem1}, so it suffices to show that
\[
\int_1^\infty\frac{dr}{p_1(r)}=\infty \quad \iff\quad \int_1^\infty\frac{dr}{p(r)}=\infty.
\]
The implication ``$\Rightarrow$'' is trivial since $p_1(r)\geq p(r)$. We now show ``$\Leftarrow$''.
If $\liminf p(r)<\infty$, then picking a sequence $r_j$ as in \eqref{auxp1} and arguing as in \eqref{auxp7}
we obtain
\[
\int_{r_{j-1}}^{r_j}\frac{dr}{p_1(r)}\geq c_\ga\frac{r_j^{\ga+1}}{r^\ga_j(p(r_j)+\la r_j)}
\geq c_\ga\frac{r_j}{C+\la r_j}\geq c_\ga',
\]
which implies the result. So we may assume that $\lim p(r)=\infty$, and using the same numbers $\ell_j$ as in the proof of Lemma \ref{Lem1} we have
\[
\int_1^\infty\frac{dr}{p_1(r)} =  \sum_{j=1}^\infty \int_{\ell_{j-1}}^{\ell_j}\frac{r^\ga dr}{r^\ga (p(r)+\la r)}
\geq  c_\ga  \sum_{j=1}^\infty  \frac{\ell_j}{p(\ell_j)+\la \ell_j},
\]
arguing in the right inequality as in \eqref{intg}. Consider the set of indices
\[
J=\{j\in\SN\mid p(\ell_j)\leq \ell_j\}.\]
If $J$ is an infinite set then
\[
\int_1^\infty\frac{dr}{p_1(r)} 
\geq  c_\ga  \sum_{j\in J}  \frac{\ell_j}{p(\ell_j)+\la \ell_j}\geq  c_\ga  \sum_{j\in J}  \frac{\ell_j}{(1+\la)\ell_j}=\infty.
\]
On the contrary, if $J$ is finite, then there exists some $j_0$ such that $p(\ell_j)>\ell_j$ for all $j\geq j_0$. Thus
\[
\int_1^\infty\frac{dr}{p_1(r)} 
\geq  c_\ga  \sum_{j\geq j_0}  \frac{\ell_j}{p(\ell_j)+\la \ell_j}\geq  \tfrac{c_\ga}{1+\la}  \sum_{j\geq j_0}  \frac{\ell_j}{p(\ell_j)}=\infty,
\]
the last equality due to \eqref{intp=sum}.
\end{proof}


\

We can now state a characterization result. The statements for $\Tou$ seem new in the literature,
while the sufficient condition for $\Th$ is a special case of the more general setting in \cite{Hayne}.

\begin{theorem} \label{UniqT}
Let $p(r)$ be positive and continuous for $r\geq1$, and assume that $r^\ga p(r)$ is non-decreasing for some $\ga\geq0$. 
Then, the following assertions are equivalent for each of the operators  $L=\Dt-2x\cdot\nabla$ and $L=\Dt-|x|^2$.

\sline (a) the equation \eqref{P} has at most one classical solution
$u$ in the class $\mathscr{C}(p)$ 

\sline (b) the function $p$ satifies
\Be
\int_1^\infty\frac{dr}{p(r)}=\infty.
\label{intbpb}
\Ee
\end{theorem}
\begin{proof} 
\indent {\it Case $L = \Dt-2x\cdot\nabla$.} Assume first that property (b) holds.
We must show that every continuous OU-temperature $U(x,t)$ in $\SR^n\times[0,T_0)$ such that $U(x,0)\equiv0$ and  which belongs to the class $\bC(p)$ is identically zero. We may assume that $T_0<\infty$, and we let $s_0=(1-e^{-4T_0})/4$. 
Using \eqref{TuU}, we consider  the function 
\[
u=T^{-1}U\in\T(\SR^n\times(0,s_0)) 
\]
which is continuous in $\SR^n\times[0,s_0)$ and satisfies $u(y,0)\equiv0$. We only need to show that $u\in\bC(p_2)$, for a suitable 
$p_2$ with 
\Be
\int_1^\infty \frac{dr}{\bp_2(r)}=\infty,
\label{intp3}
\Ee
since then T\"acklind's Theorem \ref{heat-T} will imply that $u\equiv 0$
(and hence also $U\equiv0$). 
Let $y=x/\exp(2t)$ and $s=(1-e^{-4t})/4$ with $s\in(0,s_0)$ and $|y|\geq1$. Then, in view of \eqref{TuU},  if we set $\al=\exp(2T_0)$,
we have
\begin{align*}
|u(y,s)| & = \, |U(x,t)| \, \leq \, A\,e^{e^{2t}|y|p(e^{2t}|y|)}\\ 
           & \leq \,A\,e^{\al^{\ga}|y|p(\al|y|)}\, =\, A\, e^{ |y| \, p_2(|y|)},  
\end{align*}
with $p_2(r)=\al^\ga p(\al r)$. Then $u\in \bC(p_2)$, where $p_2(r)$ satisfies the usual conditions and
 by Lemma \ref{Lem1} also \eqref{intp3}. This concludes the proof of this part.

We next show that (a) $\Rightarrow$ (b). Suppose by contradiction that $\int_1^\infty \frac{dr}{\bp(r)}<\infty$. 
Then, following T\"acklind's work \cite[(35)]{Tac} one may construct a function $u\in\T(\SR^n\times\SR)$ such that
$u(y,0)\equiv0$, $u\not\equiv0$ and
\Be
\sup_{s\in\SR}|u(y,s)|\leq \,e^{|y|\,\bp(|y|\vee 1)}, \quad \forall\;y\in\SR^n.
\label{Tacbound}
\Ee
We again use \eqref{TuU} to define the function
\Be
U=Tu \in\Tou(\SR^n\times\SR), 
\label{Utac}
\Ee
which satisfies $U(x,0)\equiv0$ and $U\not\equiv0$. 
If we show that $U\in\bC(p)$ we would reach a contradiction with (a). But this is immediate from \
\Be
|U(x,t)|  = 
\Big|u\Big(\tfrac x{e^{2t}},\tfrac{1-e^{-4t}}4\Big)\Big|  \leq \;e^{\frac{|x|}{e^{2t}}\bp(\frac{|x|}{e^{2t}}\vee 1)}
\leq \;e^{|x|\bp(|x|\vee1)},\quad t>0,\;x\in\SR^n.
\label{Utacp}
\Ee

{\it Case $L = \Dt-|x|^2$.} We show that (b) $\Rightarrow$ (a). Let $V\in\Th\cap \bC(p)$ be continuous in $\SR^n\times[0,T_0)$ 
and with $V(x,0)\equiv 0$. We must show that $V\equiv0$ in $\SR^n\times(0,T_0)$, for which we may assume that $T_0<\infty$. 
By \eqref{VU} and \eqref{H-sii-OU}, the function
\[
U(x,t)=e^{nt} e^{\frac{|x|^2}2} V(x,t)\in \Tou\cap C(\SR^n\times[0,T_0))\mand U(x,0)\equiv0.
\]
Using that $V\in\bC(p)$ we see that, for $t\in(0,T_0)$ and $|x|\geq1$,
\[
|U(x,t)|\leq e^{nT_0}\,e^{\frac{|x|^2}2}\,A\,e^{|x|\,p(|x|)}\,= \, A'\, e^{|x|\,p_1(|x|)},
\]
with $p_1(r)=r/2+p(r)$. Thus, $U\in \bC(p_1)$ and by Lemma \ref{Lem2} we have $\int_1^\infty dr/\bp_1(r)=\infty$, so
the previous case gives that $U\equiv0$ (hence also $V\equiv0$).

We finally prove the converse implication (a) $\Rightarrow$ (b), 
assuming again for contradiction that (b) fails. Consider the function $U(x,t)$ constructed in \eqref{Utac} using the T\"acklind example in \eqref{Tacbound}. 
 Arguing once again as in section \ref{S_trans}, we define
\[
V(x,t)=e^{-nt}e^{-\frac{|x|^2}2}\,U(x,t),
\]
which belongs to $\Th(\SR^n\times\SR)$ and satisfies $V(x,0)\equiv0$, $V\not\equiv0$ and
\Be
|V(x,t)| \leq  |U(x,t)| \leq\,e^{|x|\bp(|x|\vee 1)},\quad t>0, \;x\in\SR^n.
\label{Vtacp}
\Ee
 Thus $V\in \bC(p)$, which is in contradiction with (a).
\end{proof}

We now turn to Theorem \ref{th_uni1}, where we shall use an additional argument to replace the sufficient condition in \eqref{Ta}
by a unilateral bound. We need the following maximum principle, shown in \cite[Theorem III]{Hayne} in a general setting
which includes the H-operator (but not the OU-operator).

\begin{theorem}[Hayne]\label{hay}
Let $L$ be one of the operators in \eqref{ops}.
Let $U\in C^{2,1}(\SR^n\times(0,T_0))\cap C(\SR^n\times[0,T_0))$ be such that
\Benu
\item[(i)] $U_t\leq LU$ in $\SR^n\times(0,T_0)$
\item[(ii)] $U(x,0)\leq 0$, $x\in\SR^n$
\item[(iii)] $U(x,t)\leq Ae^{|x|p(|x|)}$, $|x|\geq1$, $t\in(0,T_0)$,
\Eenu
where $p(r)$ is positive continuous with $r^\ga p(r)$ non-decreasing for some $\ga\geq0$, and
\[
\int_1^\infty\frac{dr}{p(r)}=\infty.
\]  
Then, $U(x,t)\leq 0$ in $\SR^n\times[0,T_0)$.
\end{theorem}
\begin{proof}
The cases $L=\Dt$ and $L=\Dt-|x|^2$ are contained in the settings of \cite[Theorem III]{Hayne}.
So we only show how to derive the result for $L=\Dt-2x\cdot\nabla$. Suppose that $U\in\sTou$ satisfies (i)-(iii) above.
As in \S\ref{S_trans}, define 
\[
V(x,t)=e^{-nt} e^{-\frac{|x|^2}2} U(x,t),
\]
which also belongs to $\sTh$, and clearly satisfies (i)-(iii) (with $U$ replaced by $V$).
Then the H-case gives $V\leq 0$, and hence also $U\leq 0$ in $\SR^n\times[0,T_0)$.
\end{proof}

\subsection*{{\it Proof of Theorem \ref{th_uni1}}}
The statements of the converse implications are contained in the proof of Theorem \ref{UniqT}; see \eqref{Utacp} and \eqref{Vtacp},
so we only explain how to derive the direct implications. Suppose that $U\in \Tou$ (or $U\in \Th$) 
is continuous in $\SR^n\times[0,T_0)$ and
satisfies $U(x,0)=0$ and the upper condition in \eqref{uniTa}. Then, Hayne's Theorem \ref{hay}
implies that $U(x,t)\leq 0$ in all $\SR^n\times[0,T_0)$. 
That is we have
\[
U\in\Tou, \quad U\leq 0 \;\;{\rm in }\;\;\SR^n\times[0,T_0)\mand U(x,0)\equiv0.
\]
But these are the conditions of Widder uniqueness theorem for negative temperatures, 
which in the general version given by Aronson and Besala \cite{AB67} will imply that $U\equiv0$.
This also applies if $U\in\Th$, as both H and OU temperatures are covered in the setting of \cite{AB67}.
\ProofEnd

\


\section{Harnack inequalities} \label{Harnack}


A well-known use of mean value formulas is to establish Harnack-type inequalities. The procedure to do so for classical temperatures $u\in\T$ 
in a set $E\subset \SR^n\times\SR$ is explained in detail in Watson's book; see \cite[\S 1.7]{watsonbook}. A crucial step is to replace  
the weight function $K_{x,t}(y,s)=2^{-1}|x-y|^2/(t-s)^2$ in the mean value formula \eqref{clasMV}, by an expression which 
does not blow-up when $(y,s)\to(x,t)$ within the ball $\O(x,t;r)$. This can be done by regarding $u$ as a temperature in $\SR^{n+m}\times\SR$ and 
deriving a formula by the method of descent; see \cite[(1.23)]{watsonbook}. 
The procedure will change 
slightly the shape of the balls, which are  given by  $(y,s)\in \O_m(x,t;r)$ iff
\Be
|y-x|^2<2(n+m)(t-s)\ln\tfrac r{t-s},\quad s\in(t-r,t),
\label{Om1}
\Ee
and it also leads to slightly more complicated kernels
\Be
K^m_{x,t}(y,s)=c_{m}\,\Big(\frac{|x-y|^2}{(t-s)^2}+a_{n,m}\frac{\ln\tfrac r{t-s}}{t-s}\Big)\, R^m,
\label{Km}
\Ee
with $R^2=[2(n+m)(t-s)\ln\tfrac r{t-s}-|x-y|^2]/r$;
see \cite[p.19]{watsonbook}. If $m\geq 3$, the multiplication by the factor $R^m$ implies that
\Be
\sup_{r>0}\sup_{(x,t)\in\SR^{n+1}}\sup_{(y,s)\in\O_m(x,t;r)} r\,K^{m}_{x,t}(y,s) <\infty.
\label{supOm}
\Ee
Such constructions are also possible for temperatures in $\Tou$, and produce the following
improvement over Theorem \ref{main}.

\begin{theorem}\label{MVF2}
There exist sets $\Xi(x,t;r)$, for $(x,t)\in\R^{n+1}$ and $r>0$, 
and positive kernels $K^{\rm OU}_{x,t}(y,s)$ in $C^\infty\big(\R^n\times(-\infty,t)\big)$
satisfying
\Be
A_{n}=\sup_{r>0}\sup_{(x,t)\in\SR^{n+1}}\sup_{(y,s)\in\Xi(x,t;r)} r\,e^{2(n+2)s} K^{\rm OU}_{x,t}(y,s)<\infty,
\label{Arn}
\Ee
and such that Theorem \ref{MV-OU} holds for these kernels.
\end{theorem}
\begin{proof}
The same proof given in \S\ref{ProofMV}  works here replacing \eqref{utxtt} by the mean value formula in 
\cite[Theorem 1.25]{watsonbook}, based in the balls $\O_m(\tilde{x},\tilde{t};r)$ in \eqref{Om1}
and the kernels $K^m_{\tx,\tit}$ in \eqref{Km}. Taking $m=3$, and defining the kernels $K^{\rm OU}_{x,t}$ by 
the formula \eqref{KKphi}, all the assertions follow easily.
\end{proof}

From Theorem \ref{MVF2} and the arguments in \cite[\S1.7]{watsonbook} (or by a direct change of variables), one can 
obtain the following Harnack inequality.

\begin{corollary}
Let $E\subset\SR^n\times \SR$, $\mu$ a measure on $E$, and $K$ a compact set such that
\Be
K\subset \bigcup_{P\in\supp\mu} \La(P,E).
\label{KLa}
\Ee
Then there exists a constant $\kappa=\kappa(E,\mu,K)>0$ 
such that
$$ \max_{K} U \;\leq \,\kappa\, \int U d \mu, \quad \forall\;U\in \Tou(E)\quad \mbox{with }\quad U\geq0. $$
\end{corollary}

\begin{remark}{\rm 
In the special case when $\mu=\dt_{(x_0,t_0)}$ and $K\subset\La(x_0,t_0,E)$ one obtains
$$ \max_{K} U \;\leq \,\kappa\, \, U(x_0,t_0), \quad \forall\;U\in \Tou(E)\quad \mbox{with }\quad U\geq0.  $$
As it is expected from a parabolic equation, this is an \emph{interior} Harnack inequality, in the sense that
the $t$-location of the set $K$ must be \emph{strictly} below the point $(x_0,t_0)$.
}
\end{remark}

We finally state a Harnack-type inequality with no time separation, but with a different right hand side. 
The result in \eqref{mintq} below, for parameters $q\geq1$, could be obtained directly from Theorem \ref{MVF2}. We
give, however, a stronger formulation valid for all $q>0$. This formulation, for $u\in\sT$ and 
for standard cylinders $C_r(x_0,t_0)=B_r(x_0)\times(t_0-r^2,t_0)$, 
is contained in the work of Ferretti and Safonov; see \cite[Theorem 3.4]{FS01}, \cite[Thm 3.1]{fer}.
In our setting we shall consider the following ``Hermite-cylindrical'' sets
\[
\Ga_R(x_0,t_0)=\Big\{(x,t)\in\SR^{n+1}\mid \big|e^{2(t_0-t)}x-x_0\big|<R,\quad t\in(t_0-R^2,t_0)\;\Big\}, 
\]
and the measure $d\nu(x,t)=e^{-2(n+2)t}\,dx\,dt$.

\begin{corollary}\label{cor_mintq}
For every $q>0$, there exists a constant $\kappa=\kappa(q,n)>0$ such that 
such that
\Be
 \max_{\Ga_R(x_0,t_0)} U \;\leq \,\kappa\; \Big[\,\mint_{\Ga_{4R}(x_0,t_0)} U^q\,d\nu\,\Big]^{1/q}, 
\label{mintq}
\Ee
for all non-negative $U\in \sTou(E)$ and all  $\;{\overline{\Ga}}(x_0,t_0,4R)\subset E$ with $R\leq 1$.
\end{corollary}
\begin{proof}
We begin with the following observation:
if $(y_0,s_0)=\phi(x_0,t_0)$ and $r=Re^{-2t_0}$ with $R\leq 1$, then
\Be
C_r(y_0,s_0)\,\subset \,\phi\big(\Ga_R(x_0,t_0)\big)\,\subset\,C_{\la r}(y_0,s_0),
\label{Gincl}
\Ee
for some $\la<4$. To see this, first note that, if $(y,s)=\phi(x,t)$,
\[
|y-y_0|<r \quad\iff\quad \big|e^{2(t_0-t)}x-x_0\big|<re^{2t_0}=R.
\]
On the other hand
\Be
s_0-s=e^{-4t_0}(e^{4(t_0-t)}-1)/4\geq e^{-4t_0}(t_0-t),
\label{sos}
\Ee
So if $s_0-s\leq r^2$ then $t_0-t\leq R^2$, and the left inclusion in \eqref{Gincl} holds, actually for all $R>0$.
For the right inclusion we use the convexity inequality 
\Be
e^z-1\leq \la^2 z, \quad z\in(0,4), \quad \mbox{with $\la^2=(e^4-1)/4<14$.}
\label{exla}\Ee
Then, if $t_0-t\leq R^2$ and $R\leq 1$, from \eqref{sos} we see that  $s_0-s\leq (\la r)^2$.
Thus, \eqref{Gincl} holds, and from here it also follows that
\[
\nu\big(\Ga_R(x_0,t_0)\big)=\big|\phi\big(\Ga_R(x_0,t_0)\big)\big|\approx r^{n+2}\approx (Re^{-2t_0})^{n+2}.
\]
We now use  \cite[Theorem 3.4]{FS01} and the change of variables in \eqref{TuU}, to obtain
\[
\sup_{\Ga_R}U=\sup_{\phi(\Ga_R)}u\leq \sup_{C_{\la r}}u\lesssim \Big[\,\mint_{C_{4r}} u^q\,\Big]^{1/q}\lesssim \Big[\,\mint_{\Ga_{4R}(x_0,t_0)} U^q\,d\nu\,\Big]^{1/q}.
\]
\end{proof}

As a corollary, we deduce a uniqueness criterion which is slightly less restrictive than Theorem \ref{UniqT} above. 
The proof is similar to \cite[Thm 3.2]{fer}.

\begin{corollary}
Let $U\in\Tou(\SR^n\times(0,T_0))\cap C(\SR^n\times[0,T_0))$ be such that $U(x,0)\equiv0$.
Suppose that for some $q>0$ and for some $p(r)$ as in Theorem \ref{UniqT} with $\int_1^\infty dr/p(r)=\infty$
it holds
\[
\int_0^{T_0}\int_{\SR^n}|U(x,t)|^q\,e^{-|x|p(|x|\vee 1)}\,dx\,dt<\infty.
\]
Then $U\equiv0$.
\end{corollary}
\begin{proof}
We may assume that $T_0<\infty$. Observe that $U^2\in\sTou$ (by direct computation, or from Theorem \ref{MV-OU} and H\"older's inequality).
Applying Corollary \ref{cor_mintq} to $U^2$ with $q/2$, $|x_0|\geq2$ $t_0\in(0,T_0)$ and $R=1$, we obtain
\Beas
|U(x_0,t_0)|^q & \lesssim & \mint_{\Ga_1(x_0,t_0)}|U(x,t)|^q\,e^{-|x|p(|x|\vee1)}\, e^{|x|p(|x|\vee1)}\,d\nu(x,t)\\
& \leq  & C_{T_0}\,\int_0^{T_0}\int_{\SR^n}|U(x,t)|^q\,e^{-|x|p(|x|\vee 1)}\,dx\,dt\,e^{c_2|x_0|p(2|x_0|)},
\Eeas
using that for $(x,t)\in \Ga_R(x_0,t_0)$ it holds $c_1|x_0|\leq |x|\leq 2|x_0|$. Thus, $U\in\bC(p_2)$ with 
$p_2(r)=c_3p(2r)$. Since $\int_1^\infty dr/p_2(r)=\infty$, the result follows from Theorem \ref{UniqT}.
\end{proof}

\section*{Acknowledgments}
Part of this work was carried out during a visit of the second author to the Isaac Newton Institute for Mathematical Sciences, Cambridge (UK).
We wish to thank the INI for the support and hospitality during this visit.

\bibliographystyle{amsplain}

\end{document}